\RequirePackage{fix-cm}
\documentclass[smallcondensed]{svjour3}     
\smartqed  
\usepackage{graphicx}
\usepackage{graphics}
\usepackage{epsfig}
\usepackage{mathptmx}
\usepackage{amsmath}
\usepackage{amssymb}
\usepackage{graphicx}
\usepackage{color}
\usepackage{ifpdf}
\usepackage{url}

\newtheorem{assumption}{Assumption A.\!\!}

\newcommand{\aref}[1]{\textbf{\textbf{A.\ref{#1}}}}
\newcounter{algorithmenumi}
\newenvironment{algorithm}[1]
{\noindent\textbf{Algorithm }\refstepcounter{algorithmenumi}\textbf{\arabic{algorithmenumi}}. (#1).}{\textbf{End.}}
\usecounter{algorithmenumi}

\newcommand{\RR}{\mathbb{R}}
\newcommand{\abs}[1]{\left\vert#1\right\vert}
\newcommand{\set}[1]{\left\{#1\right\}}

\newcommand{\norm}[1]{\Vert#1\Vert}
\newcommand{\dnorm}[1]{\vert\!\Vert#1\Vert\!\vert}
\newcommand{\eofproof}{\hfill $\square$}

\begin{document}

\title{Path-Following Gradient-Based Decomposition  Algorithms For Separable Convex Optimization}


\author{Quoc Tran Dinh \and Ion Necoara \and Moritz Diehl}


\institute{Q. Tran Dinh and M. Diehl \at Optimization in Engineering Center (OPTEC) and Department of Electrical Engineering,
           Katholieke Universiteit Leuven, Belgium,
       \email{\{quoc.trandinh, moritz.diehl\}@esat.kuleuven.be}
        \and
           I. Necoara \at Automation and Systems Engineering Department, University Politehnica Bucharest,
           060042 Bucharest, Romania, \email{ion.necoara@acse.pub.ro}
    \and
           Q. Tran Dinh \at Department of Mathematics-Mechanics-Informatics, Vietnam National University, Hanoi, Vietnam.
}

\date{Received: date / Accepted: date}

\maketitle

\begin{abstract}
A new decomposition optimization algorithm, called \textit{path-following gradient-based decomposition}, is proposed to solve separable convex  optimization
problems. Unlike path-following Newton methods considered in the literature, this algorithm does not requires any smoothness assumption on the objective
function. This allows us to handle more general classes of problems arising in many real applications than in the path-following Newton methods. The new
algorithm is a combination of three techniques, namely smoothing, Lagrangian decomposition and path-following gradient framework. 
The algorithm  decomposes the original problem into smaller subproblems by using dual decomposition and smoothing via self-concordant barriers, updates
the dual variables using a path-following gradient method and allows one to solve the subproblem in parallel.
Moreover, the algorithmic parameters are updated automatically without any tuning strategy as in augmented Lagrangian approaches. 
We prove the global convergence of the new algorithm and analyze its local convergence rate. 
Then, we modify the proposed algorithm by applying Nesterov's accelerating scheme to get a new variant which has a better local convergence rate.
Finally, we present preliminary numerical tests that confirm the theory development.

\end{abstract}

\section{Introduction}\label{sec:intro}
Many optimization problems arising in engineering   and economics can conveniently be formulated  as \textit{Separable Convex Programming Problems} (SepCPs).
Particularly, optimization problems related to a network $\mathcal{N}(V, E)$ of $M$ agents, where $V$ denotes the set of nodes and $E$ denotes the set of edges
in the network, can be cast into  separable convex optimization problems. 
Several applications can be found in the literature such as distributed control, network utility maximization, resource allocation, machine learning and
multistage stochastic convex programming \cite{Bertsekas1989,Boyd2011,Palomar2006,Xiao2004,Zhao2005}.
Problems of moderate size or possessing a sparse structure  can be solved by standard optimization methods in a centralized setup. However, in many real
applications we meet problems, which may not be suitable to solve by standard optimization approaches or exploiting problem structures, e.g. nonsmooth separate
objective functions, dynamic structure or distributed information. In those situations, decomposition methods can be considered as an appropriate framework to
tackle those problems. Particularly, Lagrangian dual decomposition techniques are widely used to decompose a large-scale separable convex optimization problem
into smaller subproblem components, which can simultaneously be solved in a \textit{parallel manner} or in a \textit{closed form}.

Various approaches have  been proposed to solve (SepCP) in decomposition framework. 
One class of algorithms is based on Lagrangian relaxation and subgradient-type methods of multipliers \cite{Bertsekas1989,Duchi2012,Nedic2009}. 
It has been observed that subgradient methods are usually slow  and numerically sensitive to the choice of step sizes in practice \cite{Nesterov2004}. 
The second approach relies on augmented Lagrangian functions, see e.g. \cite{Hamdi2005,Ruszczynski1995}. 
Many variants were proposed to process the inseparability of the crossproduct terms in the augmented Lagrangian function in
different ways. Another research direction is based on alternating direction methods which were studied, for example, in \cite{Boyd2011,He2011a}. 
Alternatively, proximal point-type methods were extended to the decomposition framework, see, e.g. \cite{Chen1994,Necoara2008}. 
Other researchers employed interior point methods in the framework of decomposition such as \cite{Kojima1993,Necoara2009a,TranDinh2012c,Zhao2005}.

In this paper, we follow the same line of the dual decomposition framework but in a different way. 
First, we smooth the dual function by using self-concordant barriers. 
By an appropriate choice of the smoothness parameter, we show that the dual function of the smoothed problem is an approximation of the
original dual function. 
Then, we develop a new path-following gradient method for solving the smoothed dual problem. By  strong duality, we can also recover an approximate solution for
the original problem. 
Compared to the previous related methods mentioned above, the new approach has the following advantages. 
First, since a self-concordant barrier function only depends on its barrier parameter, this allows us to avoid a dependency on the diameter of the feasible set
as in prox-function smoothing techniques \cite{Necoara2008,TranDinh2012c}.
Second, the proposed method is a gradient-type scheme which allows to handle more general classes of  problems than in path-following Newton methods
\cite{Necoara2009,TranDinh2012c,Zhao2005}, in particular the  nonsmoothness of the objective function. 
Third, by smoothing via self-concordant barrier functions, if the objective function is smooth then instead of solving the primal subproblems as general convex
programs we can treat them by using optimality conditions which are equivalent to solving nonlinear systems. 
Finally, by convergence analysis, we provide an adaptive update for all the algorithmic parameters which still ensure the convergence of the new methods.

\vskip0.2cm
\noindent\textit{Contribution. }
The contribution of the paper can be summarized as follows:
\begin{itemize}
\item[(a)] We propose to use a smoothing technique via  barrier function to smooth the dual function of (SepCP) as in \cite{Kojima1993,Necoara2009,Zhao2005}.
However, we provide a new estimate for the dual function, see Lemma \ref{le:barrier_smooth}.

\item[(b)] We propose a new path-following gradient-based decomposition algorithm, Algorithm \ref{alg:gradient_method}, to solve (SepCP). 
This algorithm  allows one to solve the subproblem of each component in parallel. Moreover, all the algorithmic parameters are updated automatically without
using any tuning strategy.

\item[(c)] We prove the convergence of the algorithm and estimate its local convergence rate. 

\item[(d)] We modify the algorithm by applying Nesterov's accelerating scheme to obtain a new variant, Algorithm \ref{alg:fast_grad_alg}, which possesses a
convergence rate, i.e. $O(1/\varepsilon)$, where $\varepsilon$ is a given accuracy.
\end{itemize}
Let us emphasize the following points.
The new estimate of the dual function considered in this paper is different from the one in \cite{TranDinh2012c} which does not depend on the diameter of the
feasible set of the dual problem.
The worst case complexity of the second algorithm is $O(1/\varepsilon)$ which is much higher than in subgradient-type methods of multipliers
\cite{Bertsekas1989,Duchi2012,Nedic2009}. We notice that this convergence rate is optimal in the sense of Nesterov's optimal schemes
\cite{Nesterov2004,Nesterov2005} for this class of algorithms.
Moreover, we can choose smoothness parameter to adjust this convergence rate.
All the algorithms can be implemented in a \textit{parallel manner}.

\vskip0.2cm
\noindent\textit{Outline. }
The rest of this paper is organized as follows. 
In the next section, we state the problem formulation and  review the Lagrangian dual decomposition framework.
Section \ref{sec:dual_smooth} considers a smoothing technique via self-concordant barriers and provides an estimate for the dual function. The new 
algorithms and
their convergence analysis are presented in Sections \ref{sec:gradient_method} and \ref{sec:fast_grad_method}. 
Preliminarily numerical results are shown in the last section to verify our theoretical results.

\vskip0.2cm
\noindent\textit{Notation and Terminology. }
Throughout the paper,  we  work on the Euclidean  space $\RR^n$ endowed with an  inner product $x^Ty$ for $x, y \in \RR^n$ and the norm
$\norm{x}_2 := \sqrt{x^Tx}$. For a proper, lower semi-continuous convex function $f$, $\partial{f}(x)$ denotes the subdifferential of $f$ at $x$. 
If $f$ is concave then we also use $\partial{f}(x)$ for its super-differential at $x$. 
For any $x \in \textrm{dom}(f)$ such that $\nabla^2f(x)$ is positive definite, the local norm of a vector $u$ with respect to $f$ at $x$ is defined as
$\norm{u}_x := \left[u^T\nabla^2f(x)u\right]^{1/2}$ and its dual norm is $\norm{u}_x^{*} := \max\set{u^Tv ~|~ \norm{v}_x \leq 1} =
\left[u^T\nabla^2f(x)^{-1}u\right]^{1/2}$. It is obvious that $u^Tv \leq \norm{u}_x\norm{v}_x^{*}$. The notation $\RR_{+}$ and $\RR_{++}$ define
the sets of nonnegative and positive numbers, respectively. The function $\omega:\RR_{+}\to \RR$ is defined by $\omega(t) := t - \ln(1+t)$ and
its dual function $\omega_{*} : [0, 1)\to\RR$ is  $\omega_{*}(t) := -t - \ln(1-t)$.

\section{Separable Convex Programming Problems and Lagrangian Dual Decomposition}\label{sec:prob_state} 
A \textit{Separable Convex Programming} problem (SepCP) is typically written as follows:
\begin{equation}\label{eq:sep_cp}
\phi^{*} := \left\{\begin{array}{cl}
\displaystyle\max_{x} & \phi(x) := \displaystyle\sum_{i=1}^N\phi_i(x_i),\\
\mathrm{s.t.} &\displaystyle\sum_{i=1}^N(A_ix_i - b_i) = 0,\\
&x_i\in X_i, ~i=1,\cdots, N,
\end{array}\right.\tag{\textrm{SepCP}}
\end{equation}
where the decision variable $x := (x_1, \dots, x_N)$ with
$x_i\in\RR^{n_i}$, the function $\phi_i:\RR^{n_i} \to \RR$ is concave and the feasible set is described by the set $X := X_1\times
\cdots \times X_N$, with $X_{i}\in\RR^{n_i}$  nonempty, closed, convex sets for all $i=1,\cdots, N$. Matrix $A := [A_1,\dots, A_N]$, with
$A_{i}\in\RR^{m\times n_i}$ for $i=1,\dots, N$, $b := \sum_{i=1}^Nb_i \in\RR^m$ and $n_1 + \cdots + n_N = n$. The constraint $Ax - b = 0$ in
\eqref{eq:sep_cp} is called \textit{coupling linear constraint}, while $x_i\in X_i$ is  referred to as \textit{local constraints} of the $i$-th component
(agent).

Let $\mathcal{L}(x, y) := \phi(x) + y^T(Ax - b)$ be the partial Lagrangian function associated with the coupling constraint $Ax - b = 0$ of \eqref{eq:sep_cp}.
The dual problem of \eqref{eq:sep_cp} is written as:
\begin{equation}\label{eq:dual_prob}
g^{*} := \min_{y\in\RR^m} g(y),
\end{equation}
where $g$ is the dual function defined by:
\begin{equation}\label{eq:dual_func}
g(y) := \max_{x\in X}\mathcal{L}(x, y) = \max_{x\in X}\left\{\phi(x) + y^T(Ax-b)\right\}.
\end{equation}
Due to the separability of $\phi$, the dual function $g$ can be computed \textit{in parallel} as:
\begin{equation}\label{eq:dual_func}
g(y) = \sum_{i=1}^Ng_i(y), ~~ g_i(y) := \max_{x_i\in X_i}\left\{\phi_i(x_i) + y^T(A_ix_i - b_i)\right\}, ~i=1,\dots, N.
\end{equation}
Throughout this paper, we make the following assumptions:

\begin{assumption}\label{as:A1}
The following assumptions hold, see \cite{Ruszczynski1995}:
\begin{itemize}
\item[$\mathrm{(a)}$] The solution set $X^{*}$ of \eqref{eq:sep_cp} is nonempty.
\item[$\mathrm{(b)}$] Either $X$ is polyhedral or the following Slater qualification condition holds:
\begin{equation}\label{eq:slater_cond}
\mathrm{ri}(X)\cap\left\{x~|~ Ax - b = 0\right\} \neq\emptyset,
\end{equation}
where $\mathrm{ri}(X)$ is the relative interior of $X$.
\item[$\mathrm{(c)}$] The functions $\phi_i$, $i=1,\dots, N$, are proper, upper semicontinuous and concave and $A$ is full-row rank.
\end{itemize}
\end{assumption}
Assumption \aref{as:A1} is standard in convex optimization. Under this assumption, \textit{strong duality} holds, i.e. the dual problem \eqref{eq:dual_prob} is
also solvable and  $g^{*} = \phi^{*}$. Moreover, the set of Lagrange multipliers, $Y^{*}$, is bounded. 
However, under Assumption \aref{as:A1}, the dual function $g$ may not be differentiable. 
Numerical methods such as subgradient-type and bundle methods can be used to solve \eqref{eq:dual_prob}. 
Nevertheless, these methods are in general numerically intractable and slow.

\section{Smoothing via self-concordant barrier functions}\label{sec:dual_smooth}
In many practical problems, the feasible sets $X_i$, $i=1,\dots, N$ are usually  simple,  e.g.  box, polyhedra and ball.  Hence, $X_i$ can be  endowed with a
\textit{self-concordant barrier} (see, e.g. \cite{Nesterov1994,Nesterov2004}) as in the following assumption.

\begin{assumption}\label{as:A2}
Each feasible set $X_i$, $i=1,\dots, N$, is bounded and endowed with a self-concordant barrier function $F_i$  with the parameter $\nu_i>0$.
\end{assumption}
Note that the assumption on the boundedness of $X_i$ can be removed by assuming that the set of sample points  generated by the new algorithm described  below
is bounded.

\begin{remark}\label{re:as_A2}
The theory developed in this paper can be easily extended to the case $X_i$ given as follows (see \cite{Nesterov2011}) for some $i\in\set{1,\cdots, N}$:
\begin{equation}\label{eq:X_i}
X_i := X^c_i \cap X_i^a, ~~ X_i^a := \left\{x_i\in\RR^{n_i} ~:~ D_ix_i = d_i\right\},
\end{equation}
by applying the standard linear algebra routines, where the set $X_i^c$ has nonempty interior  and associated with  a $\nu_i$-self-concordant barrier $F_i$.
\end{remark}

Let us denote by $x_i^c$  the analytic center of $X_i$, i.e.:
\begin{equation}\label{eq:analytic_center}
x_i^c := \arg\!\!\!\!\min_{x_i \in \mathrm{int}(X_i)}F_i(x_i) ~ \forall i=1,\dots,N,
\end{equation}
where $\mathrm{int}(X_i)$ is the interior of $X_i$.
Since $X_i$ is bounded, $x^c_i$ is well-defined \cite{Nesterov2004}.
Moreover, the following estimates hold:
\begin{equation}\label{eq:F_upper}
F_i(x_i) - F_i(x_i^c) \geq \omega(\norm{x_i-x_i^c}_{x_i^c}) ~\textrm{and}~ \norm{x_i - x_i^c}_{x_i^c} \leq \nu_i + 2\sqrt{\nu_i}, ~~\forall x_i \in X_i,
~i=1,\dots, N.
\end{equation}
Without loss of generality, we can assume that $F_i(x^c_i) = 0$. Otherwise, we can replace $F_i$ by $\tilde{F}_i(\cdot) := F_i(\cdot) - F_i(x^c_i)$ for
$i=1,\dots,N$.
Since $X$ is separable, $F := \sum_{i=1}^NF_i$ is a self-concordant barrier of $X$ with the parameter $\nu := \sum_{i=1}^N\nu_i$.

Let us define the following function::
\begin{equation}\label{eq:g_yt}
g(y; t) :=  \sum_{i=1}^Ng_i(y;t),
\end{equation}
where 
\begin{equation}\label{eq:g_yt_i}
g_i(y; t) := \max_{x_i\in\textrm{int}(X_i)}\left\{ \phi_i(x_i) + y^T(A_ix_i - b_i) - tF_i(x_i)\right\},~ i=1,\dots, N, 
\end{equation}
with $t > 0$ being referred to as a smoothness parameter. 
Note that the maximum problem in \eqref{eq:g_yt_i} has a unique optimal solution, which is denoted by $x_i^{*}(y;t)$, due to the strict concavity of the
objective function. We call this problem the \textit{primal subproblem}. 
Consequently, the functions $g_i(\cdot,t)$ and $g(\cdot,t)$ are well-defined and smooth on $\RR^m$ for any $t > 0$. 
We call $g_i(\cdot; t)$ and $g(\cdot;t)$ the \textit{smoothed dual function} of $g_i$ and $g$, respectively.

The optimality condition for \eqref{eq:g_yt_i} is:
\begin{equation}\label{eq:optimality_g_yt}
0 \in \partial \phi_i(x_i^{*}(y;t)) + A_i^Ty - t\nabla F_i(x_i^{*}(y;t)), ~i=1,\cdots, N.
\end{equation}
Let us define the full optimal solution $x^{*}(y;t) :=
(x^{*}_1(y;t),\cdots, x^{*}_N(y;t))$. Since problem \eqref{eq:g_yt_i} is convex, this condition \eqref{eq:optimality_g_yt} is necessary and sufficient for
optimality. Moreover, the gradients of $g_i(\cdot;t)$ and $g(\cdot;t)$ are given by:
\begin{equation}\label{eq:grad_g}
\nabla{g}_i(y;t) = A_ix^{*}_i(y;t) - b_i,~~ \nabla{g}(y;t) = Ax^{*}(y;t) - b.
\end{equation}
If $f_i$ is differentiable for some $i\in\set{1,\cdots, N}$ then the condition \eqref{eq:optimality_g_yt} collapses to $\nabla{\phi}_i(x_i^{*}(y;t)) + A_i^Ty -
t\nabla F_i(x_i^{*}(y;t)) = 0$, which is indeed a \textit{system of nonlinear equations}.
First, we prove that $g(\cdot; t)$ is an approximation of the dual function $g(\cdot)$ for sufficiently small $t>0$.

\begin{lemma}\label{le:barrier_smooth}
Suppose that Assumptions \aref{as:A1} and \aref{as:A2} are satisfied.
Let $\bar{x}$ be a strictly feasible point to \eqref{eq:sep_cp}, i.e.  $\bar{x}\in \mathrm{int}(X)\cap\{x~|~Ax = b\}$. Then, for any $t>0$, we have:
\begin{equation}\label{eq:first_est}
g(y)-\phi(\bar{x}) \geq 0 ~~ \textrm{and} ~~ g(y;t) - \phi(\bar{x}) + tF(\bar{x}) \geq 0.
\end{equation}
Moreover, it holds that:
\begin{align}\label{eq:g_est}
g(y;t)  \leq g(y) \leq g(y;t) + t(\nu + F(\bar{x}))  + 2\sqrt{t\nu}\left[g(y;t)+tF(\bar{x})-\phi(\bar{x})\right]^{1/2}.
\end{align}
\end{lemma}

\begin{proof}
The first two inequalities in \eqref{eq:first_est} are trivial due to the definitions of $g(\cdot)$,  $g(\cdot;t)$  and the feasibility of $\bar{x}$. 
We only prove \eqref{eq:g_est}. Indeed, since $\bar{x}\in\mathrm{int}(X)$ and $x^{*}(y)\in X$, if we define $x^{*}_{\tau}(y) := \bar{x} + \tau(x^{*}(y) -
\bar{x})$, then $x_{\tau}(y)\in\mathrm{int}X$ if $\tau\in [0, 1)$. 
By applying the inequality  \cite[2.3.3]{Nesterov1994} we have: 
\begin{equation*}
F(x_{\tau}(y)) \leq F(\bar{x}) - \nu\ln(1-\tau).
\end{equation*}
Using this inequality together with the definition of $g(\cdot;t)$, the concavity of $\phi$ and $A\bar{x} = b$, we deduce:
\begin{align}\label{eq:lm35_proof1}
g(y;t) &= \max_{x\in\mathrm{int}(X)}\left\{\phi(x) + y^T(Ax-b) -  tF(x) \right\} \nonumber\\
&\geq \max_{\tau\in [0, 1)}\left\{ \phi(x_{\tau}(y)) + y^T(Ax_{\tau}(y) - b) - tF(x_{\tau}(y)) \right\} \\
&\geq \max_{\tau\in [0, 1)}\left\{ (1-\tau)\phi(\bar{x}) + \tau g(y) + t\nu\ln(1-\tau)\right\} - tF(\bar{x})\nonumber.
\end{align}
By solving the maximization problem on the right hand side of \eqref{eq:lm35_proof1} and  then rearranging the results, we obtain:
\begin{equation}\label{eq:lm35_proof2}
g(y) \leq g(y;t) + t[\nu  + F(\bar{x})] +  t\nu\Big[\ln\Big(\frac{g(y) - \phi(\bar{x})}{t\nu}\Big)\Big]_{+},
\end{equation}
where $[\cdot]_{+} := \max\set{\cdot, 0}$. 
Moreover, it follows from \eqref{eq:lm35_proof1} that:
\begin{align*}
g(y) - \phi(\bar{x}) &\leq \frac{1}{\tau}\Big[g(y;t)-\phi(\bar{x}) + tF(\bar{x}) + t\nu\ln(1+\frac{\tau}{1-\tau})\Big] \nonumber\\
&\leq \frac{1}{\tau}\Big[g(y;t)-\phi(\bar{x}) + tF(\bar{x})\Big] + \frac{t\nu}{1-\tau}. \nonumber
\end{align*}
If we minimize the right hand side of this inequality in $[0,1)$, then we get $g(y) - \phi(\bar{x}) \leq [\left(g(y;t)-\phi(\bar{x}) +
tF(\bar{x})\right)^{1/2} + \sqrt{t\nu}]^2$. Finally, we plug this inequality into \eqref{eq:lm35_proof2} to obtain: 
\begin{align*}
g(y) &\leq g(y;t) + t\nu + 2t\nu\ln\left(1 + \sqrt{\frac{[g(y;t) - \phi(\bar{x}) + tF(\bar{x}]}{t\nu}}\right)
 + tF(\bar{x}) \\
& \leq g(y;t) + t\nu + tF(\bar{x}) + 2\sqrt{t\nu}\left[g(y;t)-\phi(\bar{x}) + tF(\bar{x})\right]^{1/2}, 
\end{align*}
which is indeed \eqref{eq:g_est}. 
\eofproof
\end{proof}

\begin{remark}[\textit{Approximation of $g(y)$}]\label{re:epsilon_gap}
It follows from \eqref{eq:g_est} that $g(y) \leq (1 + 2\sqrt{t\nu})g(y;t)  + t(\nu + F(\bar{x})) + 2\sqrt{t\nu}(tF(\bar{x}) - \phi(\bar{x}))$. Hence, $g(y;t)\to
g(y)$ as $t\to 0^{+}$. Moreover, this estimate is different from the one in \cite{TranDinh2012c}, since  we do not assume that $Y$ is bounded.
\end{remark}

Next, we consider the following minimization problem,  called \textit{smoothed dual problem}:
\begin{equation}\label{eq:gm_dual_prob}
g^{*}(t) := g(y^{*}(t); t) = \min_{y\in\RR^m}g(y; t).  
\end{equation}
We denote by $y^{*}(t)$ the solution of \eqref{eq:gm_dual_prob}. The following lemma shows the main properties of the functions $g(y;\cdot)$ and $g^{*}(\cdot)$.

\begin{lemma}\label{le:properties_of_d}
Suppose that Assumptions \aref{as:A1} and \aref{as:A2} are satisfied. Then:
\begin{itemize}
\item[$(\mathrm{a)}$] The function $g(y;\cdot)$ is convex and nonincreasing in $\RR_{++}$ for a given $y\in\RR^m$.
Moreover, we have:
\begin{equation}\label{eq:g_est3}
g(y;\hat{t}) \geq g(y; t) - (\hat{t}-t)F(x^{*}(y;t)).
\end{equation}
\item[$(\mathrm{b)}$] The function $g^{*}(\cdot)$ defined by \eqref{eq:gm_dual_prob} is differentiable and nonincreasing in $\RR_{++}$.
Moreover, $g^{*}(t) \leq g^{*}$, $\lim_{t\downarrow 0^{+}}g^{*}(t) = g^{*} = \phi^{*}$ and  $x^{*}(y^{*}(t);t)$  is feasible to the
original problem \eqref{eq:sep_cp}.
\end{itemize}
\end{lemma}

\begin{proof}
We only prove \eqref{eq:g_est3}, the proof of the remainders can be found in \cite{Necoara2009,TranDinh2012c}. Indeed, since $g(y;\cdot)$ is convex and
differentiable and $\frac{dg(y;t)}{dt} = -F(x^{*}(y;t))\leq 0$, we have $g(y;\hat{t}) \geq g(y;t) + (\hat{t}-t)\frac{dg(y;t)}{dt} = g(y;t) - (\hat{t} - t)
F(x^{*}(y;t))$. 
\eofproof
\end{proof}

The statement (b) of Lemma \ref{le:properties_of_d} shows that if we find an approximate solution $y^k$ for \eqref{eq:gm_dual_prob} for sufficiently small
$t_k$, then $g^{*}(t_k)$ approximates  $g^{*}$ (recall that $g^{*} = \phi^{*}$) and $x^{*}(y^k;t_k)$ is approximately feasible to \eqref{eq:sep_cp}.

\section{Path-following gradient method}\label{sec:gradient_method}
In this section we design a path-following gradient algorithm to solve the dual problem \eqref{eq:dual_prob}, analyze the convergence of the algorithm and
estimate the local convergence rate.

\subsection{The path-following gradient scheme}
Since $g(\cdot;t)$ is strictly convex and smooth,  we can write the optimality  condition of \eqref{eq:gm_dual_prob} as:
\begin{equation}\label{eq:optimality_of_gm_dual_prob}
\nabla{g}(y; t) = 0.
\end{equation}
This equation has a unique solution $y^{*}(t)$. 

Now, for any given $x\in\mathrm{int}(X)$, $\nabla{F}(x)$ is positive definite. We introduce a local matrix norm:  
\begin{equation}\label{eq:norm_A}
\dnorm{A}_x^{*} :=  \norm{A\nabla^2 F(x)^{-1}A^T}_2,
\end{equation}
The following lemma shows a main property of the function $g(\cdot; t)$.

\begin{lemma}\label{le:Lipschitz_type}
Suppose that Assumptions \aref{as:A1} and \aref{as:A2} are satisfied. Then,  for all $t > 0$ and $y, \hat{y}\in\RR^m$, one has:
\begin{equation}\label{eq:gm_Lipschitz}
[\nabla{g}(y; t) - \nabla{g}(\hat{y}; t)]^T(y  - \hat{y}) \geq \frac{t\norm{\nabla{g}(y;t) - \nabla{g}(\hat{y};t)}_2^2}{c_A \left[ c_A + \norm{\nabla{g}(y,t) -
\nabla{g}(\hat{y};t)}_2\right]},
\end{equation}
where $c_A := \dnorm{A}^{*}_{x^{*}(y;t)}$. 
Consequently, it holds that:
\begin{equation}\label{eq:g_bound}
g(\hat{y};t) \leq g(y;t) + \nabla{g}(y;t)^T(\hat{y}-y) + t\omega^{*}(c_At^{-1}\norm{\hat{y}-y}_2),
\end{equation}
provided that $c_A\norm{\hat{y}-y}_2 < t$.
\end{lemma}

\begin{proof}
For notational simplicity, we denote by $x^{*} := x^{*}(y;t)$ and $\hat{x}^{*} := x^{*}(\hat{y};t)$.  From the definition \eqref{eq:grad_g} of $\nabla{g}(\cdot;
t)$ and the Cauchy-Schwarz inequality we have:
\begin{align}
&[\nabla{g}(y;t) - \nabla{g}(\hat{y};t)]^T(y-\hat{y}) = (y-\hat{y})^TA(x^{*}  - \hat{x}^{*}). \label{eq:lm31_est1}\\
&\norm{\nabla{g}(\hat{y};t) -  \nabla{g}(y;t)}_2  \leq \dnorm{A}^{*}_{x^{*}}\norm{\hat{x}^{*}- x^{*}}_{x^{*}}.\label{eq:lm31_est2}
\end{align}
It follows from \eqref{eq:optimality_g_yt} that $A^T(y-\hat{y}) = t[\nabla{F}(x^{*}) - \nabla{F}(\hat{x}^{*}] - [\xi(x^{*}) - \xi(\hat{x}^{*})]$, where
$\xi(\cdot)\in\partial\phi(\cdot)$.
By multiplying this  relation with $x^{*} - \hat{x}^{*}$ and then using \cite[Theorem 4.1.7]{Nesterov2004} and the concavity of $\phi$ we obtain:
\begin{align*}
(y-\hat{y})^TA(x^{*} - \hat{x}^{*}) &= t[\nabla{F}(x^{*}) - \nabla{F}(\hat{x}^{*})]^T(x^{*} - \hat{x}^{*}) - [\xi(x^{*}) - \xi(\hat{x}^{*})]^T(x^{*} -
\hat{x}^{*})\nonumber\\
&\overset{\tiny\textrm{concavity of}~\phi}{\geq} t[\nabla{F}(x^{*}) - \nabla{F}(\hat{x}^{*})]^T(x^{*} - \hat{x}^{*})\nonumber\\
&\geq \frac{t\norm{x^{*} - \hat{x}^{*}}_{x^{*}}^2}{1 + \norm{x^{*} - \hat{x}^{*}}_{x^{*}}}\nonumber\\
&\overset{\tiny\eqref{eq:lm31_est2}}{\geq}\frac{t\left[\norm{\nabla{g}(y;t) - \nabla{g}(\hat{y};t)}_2\right]^2}{\dnorm{A}^{*}_{x^{*}}
\left[\dnorm{A}^{*}_{x^{*}} + \norm{\nabla{g}(y;t) - \nabla{g}(\hat{y};t)}_2\right]}. \nonumber
\end{align*}
Substituting this inequality into \eqref{eq:lm31_est1} we obtain \eqref{eq:gm_Lipschitz}.

By the Cauchy-Schwarz inequality,  it follows from \eqref{eq:gm_Lipschitz} that $\norm{\nabla{g}(\hat{y};t) - \nabla{g}(y;t)} \leq
\frac{c_A^2\norm{\hat{y}-y}_2}{t - c_A\norm{\hat{y} - y}}$, provided that $c_A\norm{\hat{y} - y} \leq t$. Finally, by using the mean-value theorem, we have:
\begin{align*}
g(\hat{y};t) &= g(y;t) + \nabla{g}(y;t)^T(\hat{y}-y)  + \int_0^1(\nabla{g}(y + s(\hat{y}-y); t) - \nabla{g}(y;t))^T(\hat{y}-y)ds\\
&\leq g(y;t) + \nabla{g}(y;t)^T(\hat{y} - y)  + c_A\norm{\hat{y} - y}_2\int_0^1\frac{c_As\norm{\hat{y} - y}_2}{t - c_A s\norm{\hat{y} -y}_2}ds\\
&= g(y;t) + \nabla{g}(y;t)^T(\hat{y}-y)  + t\omega^{*}(c_At^{-1}\norm{\hat{y}-y}_2),
\end{align*}
which is indeed \eqref{eq:g_bound} provided that $c_A\norm{\hat{y}-y}_2 < t$.
\eofproof
\end{proof}

Now, we describe one step of the path-following gradient method for solving \eqref{eq:gm_dual_prob}. Let us assume that $y\in\RR^m$ and $t > 0$ are the values
at the current iteration, the values $y_{+}$ and $t_{+}$ at the next iteration are computed as:
\begin{equation}\label{eq:gradient_scheme}
\begin{cases}
t_{+} := t - \Delta{t},\\
y_{+} := y - \alpha \nabla{g}(y, t_{+}),
\end{cases}
\end{equation}
where $\alpha :=\alpha(y;t) > 0$ is the current step size and $\Delta{t}$ is the  decrement of the parameter $t$. 
In order to analyze the convergence of the scheme \eqref{eq:gradient_scheme}, we introduce the following notation:
\begin{equation}\label{eq:new_notation}
x^{*}_1 := x^{*}(y;t_{+}), ~~c_{A1} = \dnorm{A}_{x^{*}(y;t_{+})}^{*} ~\mathrm{and}~ \lambda_1 := \norm{\nabla g(y;t_{+})}_2.
\end{equation}
First, we prove an important property of the \textit{path-following gradient scheme} \eqref{eq:gradient_scheme}.

\begin{lemma}\label{le:gradient_method}
Under Assumptions \aref{as:A1} and \aref{as:A2}, the following inequality holds:
\begin{equation}\label{eq:gm_key_est1}
g(y_{+};t_{+}) \leq g(y;t) - \left[\alpha\lambda^2_1 -  t_{+}\omega^{*}(c_{A1}t_{+}^{-1}\alpha\lambda_1) - \Delta{t}F(x^{*}_1)\right],
\end{equation}
where $c_{A1}$ and $\lambda_1$ are defined by \eqref{eq:new_notation}.
\end{lemma}

\begin{proof}
Since $t_{+} = t - \Delta{t}$, by using \eqref{eq:g_est3} with $t$ and $t_{+}$, we have:
\begin{equation}\label{eq:lm32_proof1}
g(y;t_{+}) \leq g(y;t) + \Delta{t}F(x^{*}(y;t_{+})).
\end{equation}
Next, by \eqref{eq:g_bound} we have $y_{+} - y = -\alpha\nabla{g}(y;t_{+})$ and  $\lambda_1 := \norm{\nabla{g}(y;t_{+})}_2$. Hence, we can derive:
\begin{equation}\label{eq:lm32_proof2}
g(y_{+}; t_{+}) \leq g(y;t_{+}) - \alpha\lambda_1^2 + t_{+}\omega^{*}(c_{A1}\alpha\lambda_1t_{+}^{-1}).
\end{equation}
By plugging \eqref{eq:lm32_proof1} into \eqref{eq:lm32_proof2},  we obtain \eqref{eq:gm_key_est1}.
\eofproof
\end{proof}

\begin{lemma}\label{le:M_bound}
For any $y\in\RR^m$ and $t>0$, the constant $c_A := \dnorm{A}_{x^{*}(y;t_{+})}^{*}$ is bounded. 
More precisely, $c_A \leq \bar{c}_A := \kappa\dnorm{A}_{x^c}^{*} < +\infty$.
Furthermore, $\lambda := \norm{\nabla{g}(y; t)}_2$ is also bounded, i.e.: $\lambda \leq \bar{\lambda} := \kappa\dnorm{A}_{x^c}^{*} + \norm{Ax^c - b}_2$,
where $\kappa := \sum_{i=1}^N[\nu_i + 2\sqrt{\nu_i}]$.
\end{lemma}

\begin{proof}
For any $x\in\mathrm{int}(X)$, from the definition of $\dnorm{\cdot}^{*}_x$, we have:
\begin{align*}
\dnorm{A}^{*}_{x} &= \sup\left\{ [v^TA\nabla^2F(x)^{-1}A^Tv]^{1/2} ~:~ \norm{v}_2 = 1\right\}\\
& = \sup\left\{ \norm{u}^{*}_x ~:~ u = A^Tv, ~\norm{v}_2 = 1\right\}\\
&\leq \sup \left\{ (\nu+2\sqrt{\nu})\norm{u}^{*}_{x^c} ~:~ u = A^Tv, ~\norm{v}_2 = 1\right\}\\
&= (\nu+2\sqrt{\nu})\sup \left\{ \left[v^TA\nabla^2F(x^c)^{-1}A^Tv\right]^{1/2}, ~\norm{v}_2 = 1\right\}\\
& = (\nu+2\sqrt{\nu})\dnorm{A}^{*}_{x^c}.
\end{align*}
Here, the inequality in this implication follows from \cite[Corollary 4.2.1]{Nesterov2004}. 
By substituting $x = x^{*}(y;t)$ into the above inequality, we obtain the first conclusion.
In order to prove the second bound, we note that $\nabla{g}(y;t ) = Ax^{*}(y;t) - b$. Therefore, by using \eqref{eq:F_upper}, we can estimate:
\begin{align*}
\norm{\nabla{g}(y;t)}_2 &= \norm{Ax^{*}(y;t) - b}_2 \leq \norm{A(x^{*}(y;t) - x^c)}_2 + \norm{Ax^c - b}_2\\
&\leq \dnorm{A}^{*}_{x^c}\norm{x^{*}(y;t) - x^c}_{x^c} + \norm{Ax^c - b}_2\\
&\overset{\tiny\eqref{eq:F_upper}}{\leq} \kappa \dnorm{A}^{*}_{x^c} + \norm{Ax^c - b}_2,
\end{align*}
which is the second conclusion.
\eofproof
\end{proof}

Next, we show how to choose the step size $\alpha$ and the decrement $\Delta{t}$ such  that $g(y_{+};t_{+}) < g(y;t)$ in Lemma
\ref{le:gradient_method}. We note that $x^{*}(y;t_{+})$ is obtained by solving the subproblem \eqref{eq:g_yt_i} and the quantity $c_F := F(x^{*}(y;t_{+}))$ is
nonnegative and computable. By Lemma \ref{le:M_bound}, we see that:
\begin{equation}\label{eq:alpha_bounds}
\alpha(t) := \frac{t}{c_{A1}(c_{A1} + \lambda_1)} \geq \underline{\alpha}_0(t) := \frac{t}{\bar{c}_{A}(\bar{c}_{A} + \bar{\lambda})},
\end{equation}
which shows that $\alpha(t)$ is bounded away from zero.
We have the following estimate.

\begin{lemma}\label{le:choice_of_stepsize}
The step size $\alpha(t)$ defined by \eqref{eq:alpha_bounds} satisfies:
\begin{equation}\label{eq:gm_key_est2}
g(y_{+};t_{+}) \leq g(y;t) - t_{+}\omega\left(\frac{\lambda_1}{c_{A1}}\right) + \Delta{t}F(x^{*}_1).
\end{equation}
\end{lemma}

\begin{proof}
Let $\varphi(\alpha) := \alpha\lambda^2_1 -  t_{+}\omega^{*}(c_{A1}t_{+}^{-1}\alpha\lambda_1) - t_{+}\omega(\lambda_1c_{A1}^{-1})$. We can simplify this
function as $\varphi(\alpha) = t_{+}[u + \ln(1 - u)]$, where $u := t_{+}^{-1}\lambda_1^2\alpha + t_{+}^{-1}c_{A1}\lambda_1\alpha - c_{A1}^{-1}\lambda_1$.
The function $\varphi(\alpha) \leq 0$ for all $u$ and $\varphi(\alpha) = 0$ at $u = 0$ which leads to $\alpha := \frac{t}{c_{A1}(c_{A1}+\lambda_1)}$.
\eofproof
\end{proof}
Since $t_{+} = t - \Delta{t}$, if we choose  $\Delta{t} := \frac{t\omega(c_{A1}^{-1}\lambda_1)}{2[\omega(c_{A1}^{-1}\lambda_1) + F(x^{*}_1)]}$ then:
\begin{equation}\label{eq:gm_key_est3}
g(y_{+};t_{+}) \leq g(y;t) - \frac{t}{2}\omega(c_{A1}^{-1}\lambda_1).
\end{equation}
Therefore, the update rule for $t$ can be written as:
\begin{equation}\label{eq:t_update}
t_{+} := (1-\sigma)t, ~\mathrm{where}~ \sigma := \frac{\omega(c_{A1}^{-1}\lambda_1)}{2[\omega(c_{A1}^{-1}\lambda_1) + F(x^{*}_1)]} \in (0, 1).
\end{equation}

\subsection{The algorithm}
Combing the above analysis, we can describe the path-following gradient decomposition method  is follows:

\begin{algorithm}{\textit{Path-following gradient decomposition}}\label{alg:gradient_method}
\newline
\noindent\textbf{Initialization:}
\begin{itemize}
\item[1.] Choose an initial value $t_0 > 0$ and tolerances $\varepsilon_t > 0$ and $\varepsilon_g > 0$.
\item[2.] Take an initial point $y^0\in\RR^m$ and solve \eqref{eq:dual_func} \textit{in parallel} to obtain
$x_0^{*} := x^{*}(y^0;t_0)$.
\item[3.] Compute $c_A^0 := \dnorm{A}_{x^{*}_0}^{*}$, $\lambda_0 := \norm{\nabla{g}(y^0;t_0)}_2$, $\omega_0 := \omega(\lambda_0/c_A^0)$ and $c_F^0 :=
F(x^{*}_0)$.
\end{itemize}
\noindent\textbf{Iteration:}~ For~$k = 0, 1, \cdots$, perform the following steps:
\begin{itemize}
\item[]\textit{Step 1:} Update the barrier parameter: $t_{k+1} := t_k(1-\sigma_k)$, where $\sigma_k := \frac{\omega_k}{2(\omega_k + c_F^k)}$.
\item[]\textit{Step 2:} Solve \eqref{eq:dual_func} \textit{in parallel} to obtain $x^{*}_k := x^{*}(y^k, t_{k+1})$. 
Then, form the gradient vector $\nabla{g}(y^k;t_{k+1}) := Ax^{*}_k - b$.
\item[]\textit{Step 3:} Compute $\lambda_{k+1} := \norm{\nabla{g}(y^k;t_{k+1})}_2$, $c_A^{k+1} := \dnorm{A}_{x^{*}_k}^{*}$, $\omega_{k+1} :=
\omega(\lambda_{k+1}/c_A^{k+1})$ and $c_F^{k+1} := F(x^{*}_k)$.
\item[]\textit{Step 4:} If $t \leq \varepsilon_t$ and $\lambda_k \leq \varepsilon$, then terminate.
\item[]\textit{Step 5:} Compute the step size $\alpha_{k+1} := \frac{t_{k+1}}{c_A^{k+1}(c_A^{k+1} + \lambda_{k+1})}$.
\item[]\textit{Step 6:} Update $y^{k+1}$ as:
\begin{equation*}
y^{k+1} := y^k - \alpha_{k+1}\nabla{g}(y^k, t_{k+1}).
\end{equation*}
\end{itemize}
\end{algorithm}
\newline
The main step of Algorithm \ref{alg:gradient_method} is Step 2, where we need to solve in parallel the primal subproblems. To form the gradient vector
$\nabla{g}(\cdot, t_{k+1})$, one can compute in parallel by multiplying column-blocks $A_i$ of $A$ by the solution $x_i^{*}(y^k, t_{k+1})$. This task only
requires local information to be exchanged between the current node and its neighbors.

From the update rule \eqref{eq:t_update} of $t$ we can see that $\sigma_k\to 0^{+}$  as $F(x^{*}_k)\to\infty$. This happens when the barrier function is
approaching the boundary of the feasible set $X$. Hence, the parameter $t$ is not decreased.
Let $\bar{c}_F$ be a sufficiently large positive constant. We can modify the update rule of $t$ as:
\begin{equation}
t_{k+1} := \begin{cases}
t_k(1-\frac{\omega_k}{2(\omega_k + c_F^k)}) &\mathrm{if}~ c_F^k \leq \bar{c}_F,\\
t_k &\mathrm{otherwise},
\end{cases}
\end{equation}
In this case, the sequence $\set{t_k}$ generated by  Algorithm \ref{alg:gradient_method} might not converge to zero.
Moreover, the step size $\alpha_k$ computed at Step 5 depends on the parameter $t_k$. If $t_k$ is small then Algorithm \ref{alg:gradient_method}
makes short steps toward a solution of \eqref{eq:dual_prob}.

\subsection{Convergence analysis}
Let us assume that $\underline{t} = \inf_{k\geq 0}t_k > 0$. Then, the following theorem shows the convergence of Algorithm \ref{alg:gradient_method}.

\begin{theorem}\label{th:gm_convergence}
Suppose that Assumptions \aref{as:A1} and \aref{as:A2} are
satisfied. Suppose further that the sequence $\{(y^k, t_k, \lambda_k)\}_{k\geq 0}$ generated by Algorithm \ref{alg:gradient_method} satisfies $\underline{t} :=
\inf_{k\geq 0}\{t_k\} > 0$. Then:
\begin{equation}\label{eq:gm_convergence}
\lim_{k\to\infty}\norm{\nabla g(y^k,t_{k+1})}_2 = 0.
\end{equation}
Consequently, there exists a limit point $y^{*}$ of $\{y^k\}$ such that $y^{*}$ is a solution of \eqref{eq:gm_dual_prob} at $t=\underline{t}$.
\end{theorem}

\begin{proof}
It is sufficient to prove \eqref{eq:gm_convergence}.
Indeed, from \eqref{eq:gm_key_est3} we have:
\begin{equation*}
\sum_{i=0}^k\frac{t_{k}}{2}\omega(\lambda_{k+1}/c_A^{k+1}) \leq g(y^0;t_0) - g(y^{k+1};t_{k+1}) \leq g(y^0;t_0)- g^{*}.
\end{equation*}
Since $t_k \geq \underline{t} > 0$ and $c_A^{k+1} \leq \bar{c}_A$ due to Lemma \ref{le:M_bound}, the above inequality leads to:
\begin{equation*}
\frac{\underline{t}}{2}\sum_{i=0}^{\infty}\omega(\lambda_{k+1}/\bar{c}_A) \leq  g(y^0;t_0) - g^{*} < +\infty.
\end{equation*}
This inequality implies
$\lim_{k\to\infty}\omega(\lambda_{k+1}/\bar{c}_A) = 0$, which leads to $\lim_{k\to\infty}\lambda_{k+1} = 0$. By definition of $\lambda_k$
we have $\lim_{k\to\infty}\norm{\nabla{g}(y^k;t_{k+1})}_2 = 0$.
\eofproof
\end{proof}

\begin{remark}\label{re:step_size}
From the proof of Theorem \ref{th:gm_convergence}, we can  fix $c_A^k \equiv \bar{c} := \kappa\dnorm{A}^{*}_{x^c}$ in Algorithm \ref{alg:gradient_method}.
This value can be computed a priori.
\end{remark}

\subsection{Local convergence rate}
Let us analyze the local convergence rate of Algorithm \ref{alg:gradient_method}. 
Let $y^0$ be an initial point of Algorithm \ref{alg:gradient_method} and $y^{*}(t)$ be the unique solution of \eqref{eq:gm_dual_prob}. We
denote by:
\begin{equation}\label{eq:r0_quantity}
r_0(t) := \norm{y^0 - y^{*}(t)}_2.
\end{equation}
For simplicity of discussion, we assume that the smoothness parameter $t_k$ is fixed  at $\underline{t} > 0$ sufficiently small for all $k\geq 0$ (see Lemma
\ref{le:barrier_smooth}). The convergence rate of Algorithm \ref{alg:gradient_method} in the case $t_k = \underline{t}$ is stated in the following
lemma.

\begin{lemma}[\textit{Local convergence rate}]\label{le:convergene_rate}
Suppose that the initial point $y^0$ is chosen  such that $g(y^0;\underline{t}) - g^{*}(\underline{t}) \leq\frac{3\bar{c}_A}{2r_0(\underline{t})}$. Then:
\begin{equation}\label{eq:convergence_rate}
g(y^k;\underline{t}) - g^{*}(\underline{t}) \leq \frac{12\bar{c}_A^2r_0(\underline{t})^2}{16\bar{c}_Ar_0(\underline{t}) + 3\underline{t}k}.
\end{equation}
Consequently, the local convergence rate of Algorithm \ref{alg:gradient_method} is at least
$O\left(\frac{4\bar{c}_A^2r_0(\underline{t})^2}{\underline{t}k}\right)$.
\end{lemma}

\begin{proof}
Let $\Delta_k := g(y^{k};\underline{t}) - g^{*}(\underline{t})$ and $\underline{y}^{*} := y^{*}(\underline{t})$. Then $\Delta_k\geq 0$.
First, by the convexity of $g(\cdot;\underline{t})$ we have:
\begin{align*}
\Delta_k = g(y^{k};\underline{t}) - g^{*}(\underline{t}) \leq \norm{\nabla{g}(y^k,\underline{t})}_2\norm{y^k - \underline{y}^{*}}_2 =
\underline{\lambda}_k\norm{y^0 - \underline{y}^{*}}_2 \leq \underline{\lambda}_kr_0(\underline{t}). 
\end{align*}
This inequality implies:
\begin{equation}\label{eq:lm46_proof1}
\underline{\lambda}_k \geq r_0(\underline{t})^{-1}\Delta_k.
\end{equation}
Since $t_k=\underline{t}>0$ is fixed for all $k\geq 0$, it follows from \eqref{eq:gm_key_est1} that:
\begin{equation*}
g(y^{k+1};\underline{t}) \leq g(y^k;\underline{t}) - \underline{t}\omega(\underline{\lambda}_k/\underline{c}_A^k), 
\end{equation*}
where $\underline{\lambda}_k := \norm{\nabla{g}(y^k;\underline{t})}_2$ and $\underline{c}_A^k := \dnorm{A}_{x^{*}(y^k;\underline{t})}^{*}$. 
By using the definition of $\Delta_k$, the last inequality is equivalent to:
\begin{equation}\label{eq:lm46_proof2}
\Delta_{k+1} \leq \Delta_k - \underline{t}\omega(\underline{\lambda}_k/\underline{c}_A^k).
\end{equation}
Next, since $\omega(\tau) \geq \tau^2/2 - \tau^3/3 \geq \tau^2/4$ for all $0\leq \tau \leq 3/4$ and $\underline{c}_A^k\leq \bar{c}_A$ due to Lemma
\ref{le:M_bound}, it follows from \eqref{eq:lm46_proof1} and \eqref{eq:lm46_proof2} that:
\begin{equation}\label{eq:lm46_proof3}
\Delta_{k+1} \leq \Delta_k - (\underline{t}\Delta_k^2) / (4r_0(\underline{t})^2\bar{c}_A^2),
\end{equation}
for all $\Delta_k \leq 3\bar{c}_Ar_0(\underline{t})/4$.

Let $\eta := \underline{t}/(4r_0(\underline{t})^2\bar{c}_A^2)$. Since $\Delta_k\geq 0$,  \eqref{eq:lm46_proof3} implies:
\begin{align*}
\frac{1}{\Delta_{k+1}} \geq \frac{1}{\Delta_k(1-\eta\Delta_k)} = \frac{1}{\Delta_k} + \frac{\eta}{(1-\eta\Delta_k)} \geq \frac{1}{\Delta_k} + \eta. 
\end{align*}
By induction, this inequality leads to $\frac{1}{\Delta_k} \geq \frac{1}{\Delta_0} + \eta k$ which is equivalent to $\Delta_k \leq
\frac{\Delta_0}{1+\eta\Delta_0k}$ provided that $\Delta_0 \leq 3\bar{c}_Ar_0(\underline{t})/4$. Since $\eta :=
\underline{t}/(4r_0(\underline{t})^2\bar{c}_A^2)$, this inequality is indeed \eqref{eq:convergence_rate}.
The last conclusion follows from \eqref{eq:convergence_rate}.
\eofproof
\end{proof}

\section{Fast gradient decomposition algorithm}\label{sec:fast_grad_method}
Let us fix $t = \underline{t} > 0$. The function $g(\cdot; \underline{t})$ is convex and differentiable
but its gradient is not Lipschitz continuous, we can not apply Nesterov's fast gradient algorithm \cite{Nesterov2004} to solve
\eqref{eq:gm_dual_prob}.
In this section, we modify Nesterov's fast gradient method in order to obtain an accelerating gradient method for solving
\eqref{eq:gm_dual_prob}.

One step of the modified fast gradient method is described as follows.
Let $y$ and $v$ be given points in $\in\RR^m$, we compute new points $y_{+}$ and $v_{+}$ as follows:
\begin{equation}\label{eq:fast_grad_scheme}
\begin{cases}
y_{+} := v - \alpha \nabla{g}(v;\underline{t}),\\
v_{+} = \beta_1y_{+} + \beta_2y + \beta_3v,
\end{cases}
\end{equation}
where $\alpha := \underline{t}/(\bar{c}_A(\bar{c}_A + \lambda))$ is the step size, $\beta_1$, $\beta_2$ and $\beta_3$ are three parameters which will be chosen
appropriately.
First, we prove the following estimate.

\begin{lemma}\label{le:accelerating_scheme}
Let $\theta\in (0, 1)$ be a given parameter and $\rho := t/(2\theta\bar{c}_A^2)$. We define two vectors:
\begin{equation}\label{eq:r_r_+_quantities}
r := \theta^{-1}[v - (1-\theta)y] ~~\textrm{and}~~ r_{+} = r - \rho\nabla{g(v; \underline{t})}.
\end{equation}
Then the new point $y_{+}$  generated by \eqref{eq:fast_grad_scheme} satisfies:
\begin{equation}\label{eq:main_est}
\theta^{-2}\big(g(y_{+};\underline{t}) - \underline{g}^{*}\big) + \underline{t}^{-1}\bar{c}_A^2\norm{r_{+} - \underline{y}^{*}}_2^2 \leq
\theta^{-2}(1 - \theta)\big(g(y;\underline{t}) - \underline{g}^{*}\big)  + \underline{t}^{-1}\bar{c}_A^2\norm{r  - \underline{y}^{*}}_2^2,
\end{equation}
provided that $\norm{\nabla{g}(v;\underline{t})}_2\leq 3\bar{c}_A/4$,  where $\underline{y}^{*} := y^{*}(\underline{t})$ and $\underline{g}^{*} :=
g(\underline{y}^{*};\underline{t})$.
\end{lemma}

\begin{proof}
Since $y_{+} = v - \alpha\nabla{g(v; \underline{t})}$ and $\alpha =  \frac{\underline{t}}{\bar{c}_A(\bar{c}_A+\lambda)}$, it follows from \eqref{eq:g_bound}
that:
\begin{align}\label{eq:lm47_proof1}
g(y_{+};\underline{t}) &\leq g(v;\underline{t}) - \underline{t}\omega(\norm{\nabla{g}(v;\underline{t})}_2/\bar{c}_A).
\end{align}
Now, since $\omega(\tau) \geq \tau^2/4$ for all $0 \leq \tau \leq 3/4$, the inequality \eqref{eq:lm47_proof1} implies:
\begin{align}\label{eq:lm47_proof2}
g(y_{+};\underline{t}) &\leq g(v;\underline{t}) - \frac{\underline{t}}{4\bar{c}_A^2}\norm{\nabla{g(v; \underline{t})}}^2_2,
\end{align}
provided that $\norm{\nabla{g}(v;\underline{t})}_2\leq 3\bar{c}_A/4$.
For any $u := (1-\theta)y + \theta \underline{y}^{*}$ and $\theta \in (0, 1)$ we have:
\begin{align}\label{eq:lm47_proof2b}
g(v;\underline{t}) &\leq g(u;\underline{t}) + \nabla{g(v; \underline{t})}^T(v - u) \leq (1-\theta)g(y;\underline{t}) + \theta g(\underline{y}^{*};\underline{t})
\nonumber\\
& + \nabla{g(v; \underline{t})}^T(v - (1-\theta)y - \theta\underline{y}^{*}).
\end{align}
By substituting \eqref{eq:lm47_proof2b} and the relation $v - (1-\theta)y = \theta r$ into \eqref{eq:lm47_proof2} we obtain:
\begin{align}\label{eq:lm47_proof3}
g(y_{+};\underline{t}) & \leq (1 - \theta)g(y;\underline{t}) + \theta\underline{g}^{*} + \theta\nabla{g(v; \underline{t})}^T(r - \underline{y}^{*}) -
\frac{\underline{t}}{4\bar{c}_A^2}\norm{\nabla{g(v; \underline{t})}}^2_2 \nonumber\\
& = (1 - \theta)g(y;\underline{t}) + \theta\underline{g}^{*} + \frac{\theta^2\bar{c}_A^2}{\underline{t}} \Big[\norm{r - \underline{y}^{*}}_2^2
 - \norm{r - \frac{\underline{t}}{2\theta\bar{c}_A^2}\nabla{g(v; \underline{t})} -  \underline{y}^{*}}_2^2\Big] \nonumber\\
&= (1 - \theta)g(y;\underline{t}) + \theta\underline{g}^{*} + \frac{\theta^2\bar{c}_A^2} {\underline{t}}\Big[\norm{r-\underline{y}^{*}}_2^2
 - \norm{r_{+} - \underline{y}^{*}}_2^2\Big].
\end{align}
Since $1/\theta^2 = (1-\theta)/\theta^2 + 1/\theta$,  by rearranging \eqref{eq:lm47_proof3} we obtain \eqref{eq:main_est}.
\eofproof
\end{proof}
Next, we consider the update rule of $\theta$.
We can see from \eqref{eq:main_est} that if $\theta_{+}$ is updated such that $(1-\theta_{+})/\theta_{+}^2 = 1/\theta^2$ then $g(y_{+};\underline{t}) <
g(y;\underline{t})$. The above condition implies:
\begin{equation*}
\theta_{+} = 0.5\theta(\sqrt{\theta^2 + 4}-\theta). 
\end{equation*}
The following lemma provides an estimate for $k\geq 0$.

\begin{lemma}\label{le:update_theta}
The sequence $\set{\theta_k}_{k\geq 0}$ generated by $\theta_{k+1} := 0.5\theta_k[(\theta_k^2 + 4)^{1/2} - \theta_k]$ and $\theta_0 = 1$ satisfies:
\begin{equation}\label{eq:update_theta}
\frac{1}{2k+1} \leq \theta_k \leq \frac{2}{k+2}, ~~\forall k\geq 0.
\end{equation}
\end{lemma}

\begin{proof}
We note that $\theta_{k+1} = \frac{2}{\sqrt{1 + 4/\theta^2} + 1}$. If we define $s_k := 2/\theta_k$ then the last relation implies $\frac{2}{s_{k+1}} =
\frac{2}{\sqrt{s^2_k + 1} + 1}$, which leads to $\frac{2}{s_k + 2} < \frac{2}{s_{k+1}} < \frac{2}{s_k +1}$. Hence, $s_k + 1  < s_{k+1} < s_k + 2$. By
induction, we have $s_0 + k < s_k < s_0 + 2k$. More over, $\theta_0 = 1$, we have $s_0 = 2$. Substituting $s_0 = 2$ into the last inequalities and then using
the relation $s_k = 2/\theta_k$ we obtain \eqref{eq:update_theta}.
\eofproof
\end{proof}

By Lemma \ref{le:accelerating_scheme}, we have $r_{+} = r - \rho\nabla{g(v; \underline{t})}$ and $r_{+} = \theta_{+}^{-1}(v_{+} - (1-\theta_{+})y_{+})$.
From these relations, we deduce that:
\begin{equation}\label{eq:second_step}
v_{+} = (1-\theta_{+})y_{+} + \theta_{+}(r-\rho\nabla{g(v; \underline{t})}).
\end{equation}
Note that if we combine \eqref{eq:second_step} and \eqref{eq:fast_grad_scheme} then:
\begin{equation*}
v_{+} = (1 - \theta_{+} - \alpha^{-1}\rho\theta_{+})y_{+} - \theta^{-1}\theta_{+}(1-\theta)y + \left(\theta^{-1} + \alpha^{-1}\rho\right)\theta_{+}v.
\end{equation*}
This is the second line of \eqref{eq:fast_grad_scheme}, where $\beta_1 := 1-\theta_{+} -\rho\theta_{+}\alpha^{-1}$, $\beta_2 :=
-(1-\theta)\theta_{+}\theta^{-1}$ and $\beta_3 := (\theta^{-1} + \rho\alpha^{-1})\theta_{+}$.
By combining all the above analysis, we can describe the modified fast gradient algorithm in detail as follows:

\begin{algorithm}{\textit{Fast gradient decomposition algorithm}}\label{alg:fast_grad_alg}
\newline
\noindent\textbf{Initialization:}~Perform the following steps:
\begin{itemize}
\item[1.] Given a tolerance $\varepsilon > 0$. Fix the parameter $t$ at a certain value $\underline{t} > 0$.
\item[2.] Find an initial point $y^0\in\RR^m$ such that $\lambda_0 := \norm{\nabla{g}(y^0;\underline{t})}_2 \leq \frac{3\bar{c}_A}{4}$.
\item[3.] Set $\theta_0 := 1$ and $v^0 := y^0$.
\end{itemize}
\noindent\textbf{Iteration:}~ 
For~$k=0,1,\cdots$, perform the following steps:
\begin{itemize}
\item[]\textit{Step 1:} If $\lambda_k \leq \varepsilon$ then terminate.
\item[]\textit{Step 2:} Compute $r^k := \theta_k^{-1}[v^k - (1-\theta_k)y^k]$.
\item[]\textit{Step 3:} Update $y^{k+1}$ as:
\begin{equation*}
y^{k+1} := v^k - \alpha_k\nabla{g}(v^k, \underline{t}),
\end{equation*}
where $\alpha_k = \frac{\underline{t}}{\bar{c}_A(\bar{c}_A + \lambda_k)}$.
\item[]\textit{Step 4:} Update $\theta_{k+1} := \frac{1}{2}\theta_k[(\theta_k^2 + 4)^{1/2} - \theta_k]$.
\item[]\textit{Step 5:} Update 
\begin{equation*}
v^{k+1} := (1-\theta_{k+1})y^{k+1} + \theta_{k+1}(r^k - \rho_k\nabla{g}(v^k)), 
\end{equation*}
where $\rho_k := \frac{\underline{t}}{2\bar{c}_A^2\theta_k}$.
\item[]\textit{Step 6:} Solve \eqref{eq:g_yt_i} \textit{in parallel} to obtain $x^{*}_{k+1} := x^{*}(v^{k+1}; \underline{t})$. 
Then form a gradient vector $\nabla{g}(v^{k+1};\underline{t}) := Ax^{*}_{k+1} - b$ and compute $\lambda_{k+1}  := \norm{\nabla{g}(v^{k+1};\underline{t})}_2$.
\end{itemize}
\end{algorithm}

The core step of Algorithm  \ref{alg:fast_grad_alg} is Step 6, where we need to solve $M$ primal subproblems in parallel. 
Algorithm \ref{alg:fast_grad_alg} differs from Nesterov's fast gradient algorithm \cite{Nesterov2004} at Step 5, where $v_{k+1}$ not only
depends on $y^{k+1}$ and $v^k$ but also on $\nabla{g}(v^k;\underbar{t})$.

The following theorem shows the convergence of Algorithm \ref{alg:fast_grad_alg}.

\begin{theorem}\label{th:fast_grad_convergence}
Let $y^0\in\RR^m$ be an initial point of Algorithm \ref{alg:fast_grad_alg} such that $\norm{\nabla{g}(y^0;\underline{t})}_2 \leq \frac{3\bar{c}_A}{4}$.
Then the sequence $\{(y^k, v^k)\}_{k\geq 0}$ generated by Algorithm \ref{alg:fast_grad_alg} satisfies:
\begin{equation}\label{eq:main_est2}
g(y^k;\underline{t}) - \underline{g}^{*}(\underline{t}) \leq \frac{4\bar{c}_A^2}{\underline{t}(k+1)^2}\norm{y^0 - y^{*}(\underline{t})}^2.
\end{equation}
\end{theorem}

\begin{proof}
From \eqref{eq:main_est} and the update rule of $\theta_k$, we have:
\begin{align*}
\theta_{k}^{-2}(g(y^{k+1};\underline{t}) - \underline{g}^{*}) + \bar{c}_A^2\underline{t}^{-1}\norm{r^{k+1}-\underline{y}^{*}}^2_2 &\leq
\theta_k^{-2}(1-\theta_{k})(g(y^k;\underline{t}) - \underline{g}^{*}) + \bar{c}_A^2\underline{t}^{-1}\norm{r^k-\underline{y}^{*}}^2_2 \\
& \leq \theta_{k-1}^{-2}(g(y^k;\underline{t}) - \underline{g}^{*}) + \bar{c}_A^2\underline{t}^{-1}\norm{r^k-\underline{y}^{*}}^2_2 
\end{align*}
By induction, we obtain from this inequality that:
\begin{align*}
\theta_{k-1}^{-2}(g(y^k;\underline{t}) - \underline{g}^{*}) &\leq  \theta_0^{-2}(g(y^1;\underline{t}) - \underline{g}^{*}) +
\bar{c}_A^2\underline{t}^{-1}\norm{r^1-\underline{y}^{*}}^2_2 \\
& \leq (1-\theta_0)\theta_0^{-2}(g(y^0;\underline{t}) - \underline{g}^{*}) + \bar{c}_A^2\underline{t}^{-1}\norm{r^0-\underline{y}^{*}}^2_2, 
\end{align*}
for $k\geq 1$.
Since $\theta_0 = 1$ and $y^0 = v^0$, we have $r^0 = y^0$ and the last inequality implies $g(y^k;\underline{t}) - \underline{g}^{*} \leq
\bar{c}_A^2\theta_{k-1}^2\underline{t}^{-1}\norm{y^0 - \bar{y}}^2_2$.
Since $\theta_{k-1} \leq \frac{2}{k+1}$ due to Lemma \ref{le:update_theta}, we obtain \eqref{eq:main_est2}.
\eofproof
\end{proof}

Let us denote by:
\begin{equation}\label{eq:local_region}
\mathcal{R}(\bar{c}_A;\underline{t}) := \set{y^0\in\RR^m ~|~ \norm{\nabla{g}(y^0;\underline{t})}_2 \leq \frac{3\bar{c}_A}{4}}. 
\end{equation}
It is obvious that $y^{*}(\underline{t})\in \mathcal{R}(\bar{c}_A;\underline{t})$. This set is a neighbourhood of the solution $y^{*}(\underline{t})$ of the
problem \eqref{eq:gm_dual_prob}.

\begin{remark}\label{re:complexity}
Let $\varepsilon > 0$ be a given accuracy. If we fix the barrier parameter $\underline{t} := \varepsilon$ then the worst-case complexity of Algorithm
\ref{alg:fast_grad_alg} in the neighbourhood $\mathcal{R}(\bar{c}_A;\underline{t})$ is $O(\frac{2\bar{c}_A\underline{r}_0}{\varepsilon})$, where
$\underline{r}_0 := r_0(\underline{t})$.
\end{remark}

\begin{remark}(\textit{Switching strategy})\label{re:switching_variant}
We can combine Algorithms \ref{alg:gradient_method} and \ref{alg:fast_grad_alg}  to obtain a switching variant:
\begin{itemize}
\item First, we apply Algorithm \ref{alg:gradient_method} to find a point $\hat{y}^0\in\RR^m$ and $\underline{t} > 0$ such that
$\norm{\nabla{g}(\hat{y}^0;\underline{t})}_2\leq \frac{3\bar{c}_A}{4}$.
\item  Then, we switch to use Algorithm \ref{alg:fast_grad_alg}.
\end{itemize}
We can also replace the constant $\bar{c}_A$ in Algorithm \ref{alg:fast_grad_alg} by any upper bound $\hat{c}_A$ of $\underline{c}_k$. For instance, we
can choose $\hat{c}_A := \max\left\{\bar{c}_A, 4\norm{\nabla{g}(y^0;\underline{t})}_2/3\right\}$.
\end{remark}

%
\section{Numerical tests}\label{sec:numerical_tests}
In this section, we test the switching variant of Algorithms \ref{alg:gradient_method} and \ref{alg:fast_grad_alg} proposed in Remark \ref{re:switching_variant}
which we name by \texttt{PFGDA} for solving the following convex programming problem:
\begin{equation}\label{eq:basic_pursuit}
\begin{array}{cl}
\displaystyle\min_{x\in\RR^n} &\gamma\norm{x}_1 + f(x)\\
\mathrm{s.t.} & Ax = b, ~ l\leq x\leq u,
\end{array}
\end{equation}
where $f(x) := \sum_{i=1}^nf_i(x_i)$, and $f_i : \RR \to \RR$ is a convex function, $A \in \RR^{m\times n}$, $b \in \RR^m$ and $l, u \in \RR^n$ such that $l
\leq 0 < u$. 

We note that the feasible set $X := [l, u]$ can be decomposed into $n$ intervals $X_i := [l_i, u_i]$ and each interval is endowed with a $2$-self concordant
barrier $F_i(x_i) := -\ln(x_i-l_i) - \ln(u_i-x_i) + 2\ln((u_i-l_i)/2)$ for $i=1,\dots, n$. 
Moreover, if we define $\phi(x) := -\sum_{i=1}^n[f_i(x_i) + \gamma\abs{x_i}]$ then $\phi$ is concave and separable. 
Problem \eqref{eq:basic_pursuit} can be reformulated equivalently to \eqref{eq:sep_cp}.

The smoothed dual function components $g_i(y;t)$ of \eqref{eq:basic_pursuit} can be written as:
\begin{align*}
g_i(y;t) &=\max_{l_i < x_i < u_i}\left\{ -f_i(x_i) -\gamma\abs{x_i} + (A_i^Ty)x_i - tF_i(x_i) \right\} - b^Ty/n,
\end{align*}
for $i=1,\dots, n$. 
This one-variable minimization problem is nonsmooth but it can be solved easily. 
In particular, if $f_i$ is affine then this problem can be solved in a \textit{closed form}.
In case $f_i$ is smooth, we can reformulate \eqref{eq:basic_pursuit} into a smooth convex program by adding $n$ slack variables and $2n$ additional inequality
constraints to handle the $\norm{x}_1$ part.

We have implemented \texttt{PFGDA} in C++ running on a $16$ cores Intel \textregistered Xeon $2.7$GHz workstation with $12$ GB of RAM. 
The algorithm was parallelized by using \texttt{OpenMP}.
We terminated \texttt{PFGDA} if:
\begin{equation*}
\texttt{optim} := \norm{\nabla{g}(y^k;t_k)}_2/\max\set{1, \norm{\nabla{g}(y^0;t_0)}_2} \leq 10^{-3} ~\textrm{and} ~ t_k \leq 10^{-2}. 
\end{equation*}
We have also implemented two algorithms, namely \textit{decomposition algorithm with two primal steps} \cite[Algorithm 1]{TranDinh2012a} and
\textit{decomposition algorithm with two dual steps} in \cite[Algorithm 1]{TranDinh2012c} which we named \texttt{2pDecompAlg} and \texttt{2dDecompAlg},
respectively, for solving problem \eqref{eq:basic_pursuit} and compared them with \texttt{PFGDA}. 
We terminated \texttt{2pDecompAlg} and \texttt{2dDecompAlg} by using the same conditions as in \cite{TranDinh2012c,TranDinh2012a} with the tolerances
$\varepsilon_{\mathrm{feas}} = \varepsilon_{\mathrm{fun}} = \varepsilon_{\mathrm{obj}} = 10^{-3}$ and $j_{\max} = 3$.
We also terminated all three algorithms if the maximum number of iterations $\texttt{maxiter} := 10,000$ was reached.
In the last case we clarify that the algorithm is failed.

\vskip0.2cm
\noindent\textbf{a. Basis pursuit problem. } 
If the function $f(x) \equiv 0$ for all $x$ then problem \eqref{eq:basic_pursuit} becomes a bound constrained basis pursuit
problem to recover the sparse coefficient vector $x$ of given signals based on a transform operator $A$ and a vector of observations $b$. 
We assume that $A\in\mathbb{R}^{m\times n}$, $b\in\mathbb{R}^m$ and $x\in\mathbb{R}^n$, where $m < n$ and $x$ has $k$ nonzero elements ($k\ll n)$.

In this case, we only illustrate \texttt{PFGDA} by applying it to solve some small size test problems.
In order to generate a test problem, we generate a random orthogonal matrix $A$ and a random vector $x_0$ which has $k$
nonzero elements. Then we define vector $b$ as $b := Ax_0$.

We test \texttt{PFGDA} on the four problems such that $[m, n, k]$ are $[50, 128, 14]$, $[100, 256, 20]$, $[200, 512, 30]$ and $[500, 1024, 50]$. 
The results reported by \texttt{PFGDA} are plotted in Figure \ref{fig:basis_pursuit}.
\begin{figure}[ht!]
\vskip-0.5cm
\centerline{\includegraphics[angle=0,height=8.0cm,width=12.1cm]{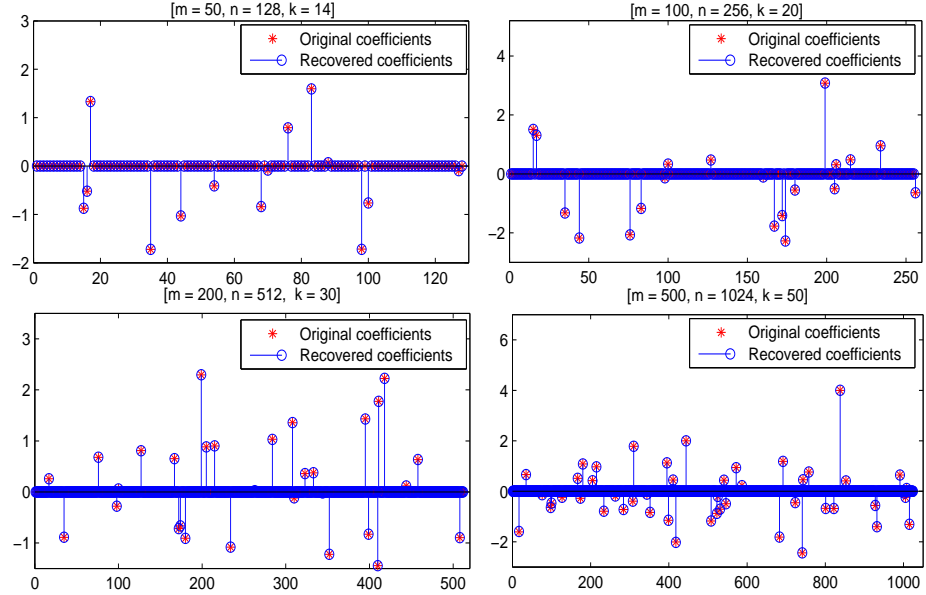}}
\caption{Illustration of \texttt{PFGDA} via the basis pursuit problem}\label{fig:basis_pursuit}
\vskip-0.5cm
\end{figure}

As we can see from these figures that the vector of recovered coefficients $x$ matches very well the vector of original coefficients $x_0$ in these four
problems. Moreover, \texttt{PFGDA} requires $376, 334, 297$ and $332$ iterations, respectively in the four problems.

\vskip0.2cm
\noindent\textbf{b. Nonlinear separable convex problems. }
In order to test the performance of \texttt{PFGDA}, we generate in this case a large test-set of problems and compare the performance of \texttt{PFGDA} with
\texttt{2pDecompAlg} and \texttt{2dDecompAlg}.

The test problems were generated as follows.
We chose the objective function $f_i(x_i) := e^{-\gamma_ix_i} - 1$, where $\gamma_i > 0$ is a given parameter for $i=1,\dots, n$.
Matrix $A$ was generated randomly in $[-1, 1]$ and then was normalized by $A/\norm{A}_{\infty}$.
We generated a sparse vector $x_0$ randomly in $[-2, 2]$ with the density $2\%$ and defined a vector $b := A\bar{x}$.
Vector $\gamma := (\gamma_1,\cdots,\gamma_n)^T$ was sparse and generated randomly in $[0, 0.5]$.
The lower bound $l_i$ and the upper bounds $u_i$ were set to $-3$ and $3$, respectively for all $i=1,\dots, n$.

We benchmarked three algorithms with performance profiles \cite{Dolan2002}. Recall that a performance profile is built based on a set $\mathcal{S}$ of $n_s$
algorithms (solvers) and a collection $\mathcal{P}$ of $n_p$ problems. Suppose that we build a profile based on computational time. We denote by
$T_{p,s} := \textit{computational time required to solve problem $p$ by solver $s$}$.
We compare the performance of algorithm $s$ on problem $p$ with the best performance of any algorithm on this problem; that is we compute the performance ratio
$r_{p,s} := \frac{T_{p,s}}{\min\{T_{p,\hat{s}} ~|~ \hat{s}\in \mathcal{S}\}}$.
Now, let $\rho_s(\tau) := \frac{1}{n_p}\mathrm{size}\left\{p\in\mathcal{P} ~|~ r_{p,s}\leq \tau\right\}$ for $\tau\in\mathbb{R}_{+}$. The function
$\rho_s:\mathbb{R}\to [0,1]$ is the probability for solver $s$ that a performance ratio is within a factor $\tau$ of the best possible ratio. We use the term
``performance profile'' for the distribution function $\rho_s$ of a performance metric.
We plotted the performance profiles in $\log$-scale, i.e. $\rho_s(\tau) := \frac{1}{n_p}\mathrm{size}\left\{p\in\mathcal{P} ~|~ \log_2(r_{p,s})\leq
\log_2\tau\right\}$.

We tested three algorithms on a collection of $50$ random problems with $m$ from $200$ to $1,500$ and $n$ from $1,000$ to $15,000$.
The profiles are plotted in Figure \ref{fig:perf_profile1}.
\begin{figure}[!ht]
\vskip-0.5cm
\centerline{\includegraphics[angle=0,height=8.0cm,width=12.1cm]{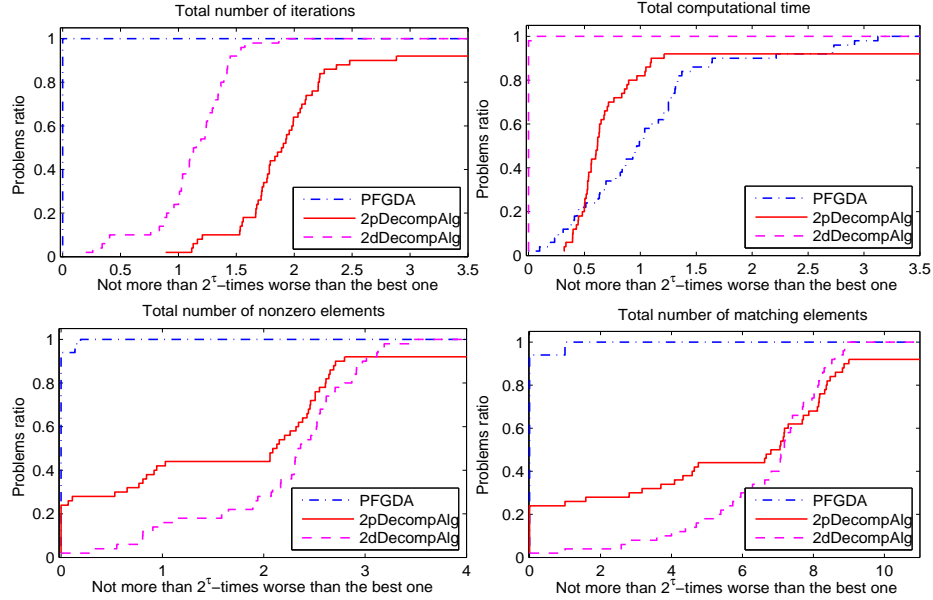}}
\caption{Performance profiles in $\log_2$ scale of three algorithms.}
\label{fig:perf_profile1}
\vskip-0.5cm
\end{figure}
Based on this test, we can make the following observations.
Both algorithms, \texttt{PDGDA} and \texttt{2dDecompAlg}, can solve all the test problems, while \texttt{2pDecompAlg} can only solve $46/50$ ($92\%$) problems.
\texttt{PFGDA} requires a significantly fewer iterations than \texttt{2pDecompAlg} and \texttt{2dDecompAlg}, and it has the best performance on $100\%$
problems in terms of number of iterations. 
\texttt{2dDecompAlg} is the best in terms of computational time where it reaches $100\%$ the test problem with the best performance. 
However, the number of nonzero elements of the obtained solution in \texttt{PFGDA} matches very well the vector of original coefficients $x_0$, while it is
rather bad in \texttt{2pDecompAlg} and \texttt{2dDecompAlg} as we can see from the last figure.
In other words, \texttt{2dDecompAlg} is not good at finding a sparse solution in this example.

\section{Concluding remarks}\label{sec:conclusion}
In this paper  we have  proposed two new dual gradient-based decomposition algorithms for solving large-scale separable convex optimization problems. We
have analyzed the convergence of these to schemes  and derived the rate of convergence. 
The first property of these methods is that they can handle general convex objective functions. Therefore, they can be applied to a wide range of
applications compared to second order methods. Second, the new algorithms can implemented in parallel and all the algorithmic parameters are updated
automatically without using any tuning strategy.
Third, the convergence rate of Algorithm \ref{alg:fast_grad_alg} is $O(1/k)$ which is optimal in the dual decomposition framework. 
Finally, the complexity estimates of the algorithms do not depend on the diameter of the feasible set as in proximal-type methods, they only depend on the
parameter of the barrier functions. 

\begin{acknowledgements}
This research was supported by Research Council KUL: PFV/10/002 Optimization in Engineering Center OPTEC, GOA/10/09 MaNet; Flemish Government:
IOF/KP/SCORES4CHEM, FWO: PhD/postdoc grants and projects: G.0320.08 (convex MPC), G.0377.09 (Mechatronics MPC); IWT: PhD Grants, projects: SBO LeCoPro; Belgian
Federal Science Policy Office: IUAP P7 (DYSCO, Dynamical systems, control and optimization, 2012-2017); EU: FP7-EMBOCON (ICT-248940), FP7-SADCO ( MC
ITN-264735), ERC ST HIGHWIND (259 166), Eurostars SMART, ACCM; the European Union, Seventh Framework Programme (FP7/2007--2013), EMBOCON, under grant agreement
no 248940; CNCS-UEFISCDI (project TE, no. 19/11.08.2010); ANCS (project PN II, no. 80EU/2010); Sectoral Operational Programme Human Resources Development
2007-2013 of the Romanian Ministry of Labor, Family and Social Protection through the Financial Agreements POSDRU/89/1.5/S/62557.
\end{acknowledgements}

\bibliographystyle{spmpsci}      


\end{document}